\documentclass{article}       % use "amsart" instead of "article" for AMSLaTeX format

\usepackage{geometry}               % See geometry.pdf to learn the layout options. There are lots.
\usepackage{graphicx}
\usepackage[usenames,dvipsnames]{xcolor}
\usepackage{amssymb,amsmath,amsthm,amsfonts}
\allowdisplaybreaks

\usepackage[english]{babel}
\usepackage[utf8]{inputenc}
\usepackage[colorlinks]{hyperref}
\hypersetup{linkcolor=blue,citecolor=blue,filecolor=black,urlcolor=blue}

\usepackage[sc]{mathpazo}
\usepackage{pgf}

\usepackage{tikz}
\newtheorem{proposition}{Proposition}[section]
\newtheorem{theorem}[proposition]{Theorem}
\newtheorem{corollary}[proposition]{Corollary}
\newtheorem{lemma}[proposition]{Lemma}

\theoremstyle{definition}

\theoremstyle{remark}
\newtheorem{remark}[proposition]{Remark}
\numberwithin{equation}{section}
\newcommand{\eps}{\varepsilon}

\newcommand{\N}{{\mathbb{N}}}
\newcommand{\Z}{{\mathbb{Z}}}

\newcommand{\R}{{\mathbb{R}}}

\newcommand{\sphere}{{\mathbb{S}}}
\newcommand{\from}{\colon}

\newcommand{\Scal}{{\mathcal{S}}}

\DeclareMathOperator{\spann}{span}

\title{Selection of the angular speed of rotating waves in segregated reaction-diffusion systems with asymmetric competition}

\author{Giuseppe Spadaro, Gianmaria Verzini and Alessandro Zilio}

\begin{document}

\maketitle

% ABSTRACT

\begin{abstract}
We investigate the existence of segregated rotating waves, arising in the singular limit 
of competition-diffusion systems of the type
\[
\partial_t u_i -\partial_{xx} u_i = f(u_i)-\beta u_i \sum_{j \neq i} a_{ij} u_j,\qquad x\in\mathbb{S}^1,\ t>0, 
1\le i,j\le k,
\]
as $\beta\to+\infty$. Here $k\ge3$, the reaction $f$ is of Fisher-KPP (logistic) type, and the competition  
coefficients $a_{ij}>0$ are not necessarily symmetric.  Assuming that, for every $i$,
\[
\dfrac{a_{i+1,i}}{a_{i,i+1}}=\lambda>0,
\]
we provide a complete characterization of the rotating waves enjoying an equivariant structure, where each 
density is a suitable rotation of any other one: such waves exist if and only if $\lambda$ belongs to an 
explicit range, in which case the angular velocity $\omega=\omega(\lambda)$ is uniquely prescribed, as is the 
rotating profile. In particular, stationary solutions (with $\omega=0$) 
exist only in the symmetric case $\lambda=1$. This marks a strong difference with the same problem with 
either Dirichlet or Neumann boundary conditions, where it is known that no periodic in time solution exists, 
also in the asymmetric case, sheding more light on some conjectures and open problems concerning the long time 
behavior of competition-diffusion systems.
\end{abstract}%
\noindent
{\footnotesize \textbf{AMS-Subject Classification}.}
{\footnotesize primary 35B25, 35B36, 34B08; secondary 35K51, 92D25.}

\section{Introduction}

This paper deals with existence, uniqueness and qualitative properties of rotating wave solutions 
to the singular limit system, arising in competitive reaction-diffusion models, when the interspecific 
competition rates become infinite. In particular, we show how the asymmetry of 
competition uniquely prescribes the speed of rotation in a rigid way. We first describe our main 
motivations, together with some consequences of our results, in Section \ref{sec:stateoftheart}; next, in Section 
\ref{sec:mainresults}, we specify the general assumptions we work with, and we state our main results.

\subsection{State of the art and motivations}\label{sec:stateoftheart}

Let us consider the following system of $k\ge 2$ nonnegative population densities $u_i=u_i(x,t)$, diffusing and interacting in 
a spatial domain $\Omega\subset\R^N$, $N\ge1$: 
\begin{equation}\label{eq:beta_sys}
        \partial_t u_i -\Delta u_i = f_i(u_i)-\beta u_i \sum_{j \neq i} a_{ij} u_j \qquad \text{ in } 
        \Omega\times(0,+\infty),\qquad 1\le i,j \le k,
\end{equation}
complemented with appropriate initial and boundary conditions. Here the reaction terms $f_i$ describe the 
intraspecific dynamics, while  $\beta>0$, $a_{ij}>0$, whence the competitive nature of the interactions. For system \eqref{eq:beta_sys}, the singular limit of strong competition is 
obtained by letting the competition parameter $\beta$ go to $+\infty$, while the coefficients $a_{ij}$ are 
kept fixed. In this case, the densities converge (uniformly and in suitable Sobolev norm) to limit 
functions $u_i$ which are \emph{segregated}, i.e.~they satisfy 
\[
u_i u_j\equiv0\qquad\text{ as }j\neq i,
\] 
hence determining a partition of the spatial domain. After seminal papers by Dancer and Du \cite{MR1303035,MR1312772,MR1312773}, 
the regularity 
and qualitative properties of the limit densities, as well as those of the emerging (possibly moving) 
free boundaries, have been the object of an intensive study in the last decades, see \cite{MR2151234,MR2146353,MR2283921,MR2384550,MR2529504} for the stationary case and \cite{MR1900331,MR2595199,MR2831712,MR2863857,MR2846360} for the evolutive one. In particular, it is known that  densities of the 
segregated limit are H\"older continuous and satisfy, in a weak sense, the system of $2k$ differential inequalities
\begin{equation}\label{eq:S}
    \partial_t u_i -\Delta u_i \le f_i(u_i),\quad
%\end{equation*}
%\begin{equation*}
    (\partial_t -\Delta)\left(u_i - \sum_{j \neq i} \frac{a_{ij}}{a_{ji}} u_j\right) \ge f_i(u_i) - \sum_{j \neq i} \frac{a_{ij}}{a_{ji}} f_j(u_j),\quad 1\le i,j \le k.
\end{equation}
In turn, \eqref{eq:S} entails the validity of the equation 
\[
\partial_t u_i -\Delta u_i = f_i(u_i)
\]
in the (open) space-time region where $u_i>0$ (and $u_j\equiv0$, $j\neq i$), and of the transmission condition 
\[
|\nabla u_i| = \frac{a_{ij}}{a_{ji}} |\nabla u_j|
\]
at points of the free boundary where only $u_i$ and $u_j$ meet (i.e.~$u_l$ is locally zero for every $l\neq i,j$).

Generally speaking, the free boundary splits into a regular part, where only two densities meet and the above transmission condition is satisfied, and a singular 
one, where three or more species interact. It is known that the description of the singular part of the free 
boundary changes drastically, depending on the competition coefficients $a_{ij}$. 
When the interaction is \emph{symmetric}, i.e.~$a_{ij}=a_{ji}$ for every $i\neq j$, the problem can be seen as 
that of a harmonic map into a singular stratified target (namely, solution vectors with only one nontrivial 
component), leading to conical singularities. On the other hand, in the \emph{asymmetric case}, which is of 
great interest for the applications, spiral patterns may arise: we refer to \cite{MR4020313}, which  contains 
a complete description of stationary spiral patterns in planar domains without reaction, and to the recent paper 
\cite{MR4877262}, where rotating spiral patterns have been constructed, in the model case of linear reactions 
and forced boundary conditions. For internal reactions of logistic type, rotating spiralling patterns have been 
detected numerically for three competing densities in \cite{MR2776460}, and numerical simulations suggest that  
such patterns arise, and are dynamically stable, only for some specific angular velocity.

The literature concerning these themes is huge and covers several different aspects, such as the 
classification/genericity of planar stationary configurations \cite{MR3948936,MR4298757,lanzara2024}, different 
bifurcation/continuation analysis \cite{MR4061643,MR4860862,MR5057427}, various nonlocal effects 
\cite{MR3259557,MR3730508,MR5049527} and the influence of the presence of a prey \cite{MR3864295,MR4026185}. 
In this paper we deal with a less studied issue, that is, the existence of limit segregated profiles 
enjoying additional properties related to the long time behavior of the evolutive solutions.

Actually, very little is known about the long time behavior of evolutive segregated solutions of 
\eqref{eq:S}, in dimension $N\ge2$; on the other hand, such topic has been successfully attacked 
in one space dimension, taking $\Omega=(0,L)$. This problem has been completely solved, in the case of 
logistic reactions, with either Dirichlet or Neumann boundary conditions on $\partial\Omega=\{0,L\}$, by 
Dancer, Wang and Zhang in \cite{MR2846360}: indeed, in such case, it is possibile to rule out the 
presence of singular points, up to a finite number of time instants $(T_l)_{l=1}^m$; moreover, up to 
relabelling the densities and discarding those which are identically zero in each $(T_l,T_{l+1})$, they 
prove that \eqref{eq:S} can be rewritten as
\begin{equation}\label{eq:S_DWZ}
\begin{cases}
\partial_t u_i -\partial_{xx} u_i = f_i(u_i)  \qquad \text{ for }\alpha_{i-1}(t) < x < \alpha_i(t),\ T_{l}< t < T_{l+1},\smallskip\\
u_i>0  \ \text{ for } \alpha_{i-1}(t) < x < \alpha_i(t),\qquad  u_i \cdot u_j \equiv 0 \text{ for }j\neq i,
\medskip\\
-\partial_x u_i = \dfrac{a_{i,i+1}}{a_{i+1,i}} \partial_x u_{i+1} \text{ at }x = \alpha_i(t).
\end{cases}
\end{equation}
Here the (increasing in $l$) 
\begin{figure}[t]
\begin{center}
\begin{tikzpicture}
    \draw[] (-.5,0) -- (8.5,0);
    \draw[black!40] (0,-.1) node[black, below]{\small $0$} -- (0,.1);
    \draw[black!40] (3,-.1) node[black, below]{\small $\alpha_{i-1}(t)$} -- (3,.1);
    \draw[black!40] (6,-.1) node[black, below]{\small $\alpha_{i}(t)$} -- (6,.1);
    \draw[black!40] (8,-.1) node[black, below]{\small $L$} -- (8,.1);
    \draw[->, black!60] (2,3.5) node[black, above]{\small $\partial_t u_i -\partial_{xx} u_i = f_i(u_i)$}
    -- (3.6,2);
    \draw[->, black!60] (7,3.1) node[black, above]{\small $-\partial_x u_i = \frac{a_{i,i+1}}{a_{i+1,i}} \partial_x u_{i+1}$}
    -- (6.1,.2);
    \draw[thick] (2.4,.5) edge[dashed, out=-20, in=150] (2.73,.35
    ) (2.73,.35) to[out=-30, in=110] (3,0) to[out=40, in=180] (4.7,2.6) to[out=0, in=100] (6,0)
     to[out=30, in=200] (6.3,0.15) edge[dashed, out=20, in=185] (6.8,0.21) ;
     \node[above] at (4.7,2.6) {\small $u_i$};
     \node[above left] at (2.4,.5) {\small $u_{i-1}$};
     \node[above right] at (6.8,0.21) {\small $u_{i+1}$};
\end{tikzpicture}
\caption{qualitative representation of solutions of system \eqref{eq:S_DWZ}.}
\label{fig:DWZ}
\end{center}
\end{figure}
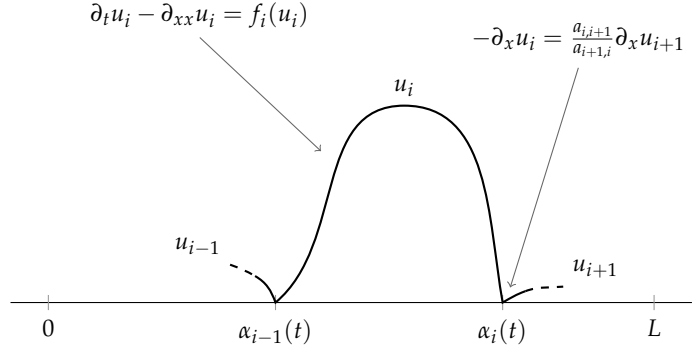
time instants $T_l$ are those in which one or more densities become extinct, with $T_0=0$ and $T_{m+1}=+\infty$, the $C^1$ functions
\[
0=\alpha_0<\alpha_1(t)<\alpha_2(t)<\dots<\alpha_{k_l-1}(t)<\alpha_{k_l}=L,\qquad T_{l}< t < T_{l+1},
\]
locate the (finite number of) points of the free boundary at time $t$, all regular, and $k_l$ is the 
(decreasing in $l$) number of (components of) densities which are nontrivial on $(T_l,T_{l+1})$ (see Figure \ref{fig:DWZ}). This description brings much insight into the problem, regardless of the symmetry properties 
of the competition coefficients $a_{ij}$: indeed, on the one hand, in the symmetric case $a_{ij}=a_{ji}$ for every $i\neq j$, the transmission conditions imply that the function
\[
v = u_1 - u_2 + u_3 - \dots - (-1)^{k_l}u_{k_l}
\]
is $C^1$ in $[0,L]$ and it satisfies a suitable equation; on the other hand, the same can be inferred 
in the asymmetric case for the function
\begin{equation}\label{eq:combin}
\tilde v = u_1 - \dfrac{a_{21}}{a_{12}} u_2 + \dfrac{a_{21}a_{32}}{a_{12}a_{23}} u_3 - \dots - (-1)^{k_l} \prod_{i=1}^{k_l-1} \dfrac{a_{i+1,i}}{a_{i,i+1}}\, u_{k_l}.
\end{equation}
Using this characterization, complemented with suitable initial data and either Dirichlet or Neumann 
boundary conditions for $u_1$ at $x=0$ and $u_{k_l}$ at $x=L$, \cite{MR2846360} assures that 
there exists a unique solution of \eqref{eq:S_DWZ}, which converges, for $t\to+\infty$, to one of the finitely many steady 
states of the system.

As a consequence, system \eqref{eq:S_DWZ} has simple dynamics, in particular, has 
no periodic solution. As noticed in \cite{MR2846360}, this is relevant in relation with some conjectures and open 
problems raised by E.N. Dancer and Weiming Ni, concerning the long time behavior of both \eqref{eq:S} and 
\eqref{eq:beta_sys}.  

Motivated by the above analysis, and by the already mentioned existence of planar, rigidly rotating 
waves with spiral interfaces, obtained for linear reactions in \cite{MR4877262}, the following natural 
questions arises: considering system \eqref{eq:S_DWZ} in the one dimensional torus 
\[
\Omega=\mathbb{S}^1 =\mathbb{T}^1
\]
(or, equivalently, in $\Omega=(0,L)$ with periodic boundary conditions), does it admit simple dynamics 
as in the case of Dirichlet/Neumann boundary conditions? Or maybe more complex solutions arise, such as 
periodic in time solutions in the form of rotating waves? 

To analyse this question, we first notice that, with periodic conditions, one main tool from \cite{MR2846360} is no longer available, as, even in the case $u_1(0)=u_{k_l}(L)=0$,  
\eqref{eq:combin} could not be $C^1$ on $\sphere^1$ unless
\begin{equation}\label{eq:asymm_prod}
\prod_{i=1}^{k_l-1} \dfrac{a_{i+1,i}}{a_{i,i+1}} = \dfrac{a_{k_l,1}}{a_{1,k_l}}.
\end{equation}
Actually, this last condition can be seen as a generalization of the symmetry condition $a_{ij}=a_{ji}$ 
(which implies it), and in fact, in higher dimension, spirals appear only if \eqref{eq:asymm_prod} is violated \cite{MR4020313,MR4877262}; notice that, in particular, \eqref{eq:asymm_prod} is always satisfied for $k=2$ densities, thus $k\ge 3$ densities are needed to violate this condition. Moreover, since $\sphere^1$ is shift invariant, steady states of \eqref{eq:S_DWZ} are no longer 
isolated, as they come in circles (apart from the constant ones). Nonetheless, this latter fact is not 
by itself an obstruction to simple dynamics, at least for problems with fewer degrees of freedom: more 
precisely, it was proved by Matano in \cite{MR956085} that, roughly, any bounded solution of 
\[
\begin{cases}
\partial_t u -\partial_{xx} u = f(u) & x\in\sphere^1,\ t>0\\
u(x,0)=u_0(x) & x\in\sphere^1,
\end{cases}
\]
converges to a steady state as $t\to+\infty$. This is due to the fact that, in dimension 1, the number of 
zeros of a solution at time $t$ can be used as a discrete Lyapunov function. As a matter of fact, Matano's argument can be merged with  
the strategy in \cite{MR2846360} to show that, when \eqref{eq:asymm_prod} is satisfied, solutions of 
\eqref{eq:S_DWZ} converge to steady states also in the case of periodic conditions, as it is going to be proved in \cite{boselli}. 
Hence, under \eqref{eq:asymm_prod}, system \eqref{eq:S_DWZ} has simple dynamics also in the case of periodic boundary conditions, while the above questions are still open when \eqref{eq:asymm_prod} fails. 

The main aim of this paper is to provide some insight into this problem, by constructing periodic 
rotating solutions of \eqref{eq:S_DWZ} on $\sphere^1$, under some structural assumptions which imply that 
\eqref{eq:asymm_prod} fails. Moreover, we are going to show that the speed of rotation $\omega$ of such 
solutions is rigidly prescribed by the asymmetric nature of the competition parameters, and in particular 
steady states (i.e.~rotating waves with speed $\omega=0$) exist if and only if the interaction satisfy 
\eqref{eq:asymm_prod}. This shows that, passing from Dirichlet/Neumann to periodic conditions, if 
\eqref{eq:asymm_prod} fails, then \eqref{eq:S_DWZ} may admit periodic in time solutions, while nontrivial 
steady states may not exist. Therefore, passing to periodic boundary conditions, the 
asymmetric nature of the competition coefficients (irrelevant in the Dirichlet/Neumann case) gains a 
prominent role in the long time dynamics, and as a byproduct we identify a sort of dynamical Turing effect, 
triggered by such asymmetry.

For our construction, we work in a model case, where the $k$ densities enjoy an equivariant cyclic 
structure, so that we can look for solutions where each density is a suitable fixed rotation of any other 
one. For concreteness, given $R>0$ and 
$k\ge3$, let us consider system \eqref{eq:S_DWZ} in $\Omega = \R/(2\pi R\Z)$ (i.e.~the $k$ densities $u_i$ 
are $2\pi R$-periodic in space), in the case
\begin{equation}\label{eq:cyclic_data}
f_i(s)=f(s), \qquad \dfrac{a_{i+1,i}}{a_{i,i+1}}=\lambda>0, \qquad \text{ for every }1\le i \le k
\end{equation}
(understanding $\frac{a_{k+1,k}}{a_{k,k+1}}=\frac{a_{1,k}}{a_{k,1}}$). With this assumption, 
\eqref{eq:asymm_prod} rewrites as
\[
\lambda^{k-1}=\lambda,
\]
and the parameter $\lambda$ encodes the asymmetric nature of the competition coefficients $a_{ij}$: the 
symmetric case corresponds to $\lambda=1$, while \eqref{eq:asymm_prod} fails if and only if 
$\lambda\neq1$. In this setting, it makes sense to look for rotating wave 
solutions satisfying, for some fixed angular speed $\omega\in\R$, and $1\le i \le k$,
\begin{equation}\label{eq:cyclic_sols}
\alpha_i(t)= \omega t + i \frac{2 \pi R}{k},\qquad\qquad u_i(x,t)=U(x - \alpha_i(t)).
\end{equation}
It is easy to check that, under assumptions \eqref{eq:cyclic_data}, the densities in \eqref{eq:cyclic_sols} 
solve \eqref{eq:S_DWZ} (with $t\in\R$) if and only if the real variable profile $U$ solves, up to translations, 
the overdetermined problem
\begin{equation}\label{eq:U_coro}
    \begin{cases}
        -U'' -\omega U'= f(U) & \text{in } \left(-\frac{ \pi R}{k},\frac{ \pi R}{k} \right)\smallskip\\
        U > 0 & \text{in } \left(-\frac{ \pi R}{k},\frac{ \pi R}{k} \right)\smallskip\\
        U\left(-\frac{ \pi R}{k}\right)=U\left(\frac{ \pi R}{k}\right)=0\smallskip\\
        U'\left(-\frac{ \pi R}{k}\right)+\lambda U'\left(\frac{ \pi R}{k}\right)=0
    \end{cases}
\end{equation}

We leave to the next section the statement of our main results, in the full generality; we conclude 
this section with a partial consequence of them, which helps in highlighting the scope 
of our analysis, in view of the previous discussion.
\begin{theorem}\label{thm:coro}
Let us assume \eqref{eq:cyclic_data}, with the logistic-type reaction
\[
f(s)=as-bs^p,\qquad\text{where $a,b>0$ and $p>1$.}
\]
Then system \eqref{eq:S_DWZ} admits a rotating wave solution, as in \eqref{eq:cyclic_sols}, \eqref{eq:U_coro}, 
if and only if 
\[
3\le k <4R^2a\qquad\text{ and }\qquad \exp \left(-\pi\sqrt{\frac{4R^2a}{k}-1}\right)
<\lambda< \exp \left(\pi\sqrt{\frac{4R^2a}{k}-1}\right).
\] 
Moreover, for every $\lambda$ in this range, the angular velocity $\omega$ is uniquely prescribed, and the rotating wave 
solution is unique up to translations.  

In particular, $\omega=0$ if and only if $\lambda=1$, so that in the asymmetric case no nontrivial steady 
state exists, while in the symmetric one any rotating wave is just a steady state. 
\end{theorem}
\begin{remark}\label{rmk:its_a_coro}
Although it is stated as a theorem, the result above is an easy corollary of our main results Theorems 
\ref{thm:main1}, \ref{thm:main2} ahead. On the other hand, by direct computations (using the substitution 
$u_i(x,t)=h_i\tilde u_i(k_i x,t+\xi_i)$, for suitable choices of $h_i$, $k_i$ and $\xi_i$), one can 
obtain a result like Theorem \ref{thm:coro} in a less equivariant setting (at least apparently), i.e.~when the segregated densities $\tilde u_i$ satisfy the equations
\[
\partial_t \tilde u_i -d_i\partial_{xx} \tilde u_i = a_i\tilde u_i-b_i\tilde u_i^p
\]
on the interval of length $\frac{2\pi R d_i}{k}$ where they are positive, which moves at speed $\omega$, 
the parameters are such that
\[
d_ia_i = a, \qquad \dfrac{a_{i+1,i}}{a_{i,i+1}}\cdot \left(\dfrac{d_i}{d_{i+1}}\right)^{\frac{p}{p-1}}\cdot \left(\dfrac{b_i}{b_{i+1}}\right)^{\frac{1}{p-1}}=\lambda, \qquad\qquad\text{ for every }i,
\]
and the structure is periodic in space of period
\[
L = \frac{2\pi R }{k}\,\sum_{i=1}^kd_i.
\]

Moreover, although our solutions are intrinsically one dimensional, they can be readily seen as 
one dimensional solutions in higher dimensional space domains.
\end{remark}

\subsection{Main results}\label{sec:mainresults}

To describe our main results, we first state our assumptions on the nonlinearity $f$, which is 
of Fisher-KPP (logistic) type, modeled on $f(u)=a u - b u^p$, with $a,b>0$ and $p>1$. 

As a more general nonlinearity, 
throughout the paper we assume that $f$ is of class $C^1$ and that there exists $\bar u>0$ 
such that
\begin{equation}\label{eq:HP_f_base}
\begin{cases}
f(0)=f(\bar u)=0\\ f(s)>0\text{ for }0<s<\bar u,\\ f(s)<0\text{ for }s>\bar u,
\end{cases}
\qquad\text{and we write }a:=f'(0);
\end{equation}
finally, we assume the following enhanced Fisher-KPP property: the map
\begin{equation}\label{eq:HP_f_enhancedKPP}
%f'(s_1) -\frac{f(s_1)}{s_1}>f'(s_2) -\frac{f(s_2)}{s_2}\qquad\text{ for }0<s_1<s_2<\bar u.
s\mapsto f'(s) -\frac{f(s)}{s}\qquad\text{ is strictly decreasing in }(0,\bar u).
\end{equation}
Some choices in our assumptions are clarified in the following two remarks.
\begin{remark}\label{rmk:HP2}
We remark that, since $\lim_{s\to0^+} \frac{f(s)}{s}=f'(0)$, assumption \eqref{eq:HP_f_enhancedKPP} entails
\begin{equation}\label{eq:HP_f_standardKPP}
f'(s) -\frac{f(s)}{s}<0\quad\text{ in }(0,\bar u), 
\end{equation}
which in turn implies some more standard Fisher-KPP-like conditions, such as:
\begin{equation}\label{eq:HP_f_monoto}
%f'(s_1) -\frac{f(s_1)}{s_1}>f'(s_2) -\frac{f(s_2)}{s_2}\qquad\text{ for }0<s_1<s_2<\bar u.
\frac{f(s)}{s}\text{ is strictly decreasing in }(0,\bar u),
\qquad\text{ whence }f'(0)=:a>0,\ f(s)<as\text{ for }s>0,
\end{equation}
see e.g.~\cite{Kolmogorov1937,MR511740,MR2214420}. 
To emphasize the role of the parameter $a$, which is central in our discussion, it is convenient to 
introduce the notation 
\begin{equation}\label{eq:f_g}
  f(u)=au - g(u),\qquad\text{ whence }g(0)=g'(0)=0,\ g(s)>0\text{ for }s>0.
\end{equation}
With this notation, \eqref{eq:HP_f_standardKPP} and \eqref{eq:HP_f_enhancedKPP} rewrite as
\begin{equation}\label{eq:HP_g_standardKPP}
g'(s) -\frac{g(s)}{s}>0\quad\text{ in }(0,\bar u), 
\end{equation}
\begin{equation}\label{eq:HP_g_enhancedKPP}
%f'(s_1) -\frac{f(s_1)}{s_1}>f'(s_2) -\frac{f(s_2)}{s_2}\qquad\text{ for }0<s_1<s_2<\bar u.
g'(s) -\frac{g(s)}{s}\qquad\text{ is strictly increasing in }(0,\bar u),
\end{equation}
respectively.\\ 
It is worth mentioning that, while most of our proofs are carried out just using the weaker condition  
\eqref{eq:HP_g_standardKPP}, the stronger one \eqref{eq:HP_g_enhancedKPP} is essential in our proof of the 
crucial Lemma \ref{lem:sign_h_r} ahead.  
At this stage we do not know if such assumption is 
technical or substantial, and it is an interesting open problem whether it can be removed or not. 
For the model nonlinearity $f(u)=a u - b u^p$, with $a,b>0$ and $p>1$, we have
\[
\frac{g(s)}{s} = b s^{p-1},\qquad g'(s) -\frac{g(s)}{s} = b(p-1) s^{p-1},
\]
thus the two assumptions are essentially the same, and they are both satisfied. On the other hand, it is not difficult to construct 
nonlinearities $f$ satisfying \eqref{eq:HP_f_standardKPP} but not \eqref{eq:HP_f_enhancedKPP}.
\end{remark}
\begin{remark}\label{rmk:HP1}
The part of assumption \eqref{eq:HP_f_base} concerning the behavior of $f(s)$ for $s>\bar s$ 
is necessary only for the part of our results concerning non-existence or uniqueness. For the 
existence ones, it is enough to assume $f\in C^1([0,\bar u))$, 
$f(0)=\lim_{s\to \bar u}f(s)=0$, $f(s)>0$ for $s\neq0$. 
\end{remark}

By the discussion in Section \ref{sec:stateoftheart}, leading to the overdetermined problem \eqref{eq:U_coro}, and under the previous assumptions on the reaction $f$, 
from now on we focus on finding solutions of the problem
\begin{equation}\label{eq:u}
    \begin{cases}
        -u'' -\omega u'= f(u) = au - g(u)& \text{in } (-\alpha,\alpha)\\
        u > 0 & \text{in } (-\alpha,\alpha)\\
        u(-\alpha)=u(\alpha)=0,
    \end{cases}
\end{equation}
subject to the overdetermined condition 
\begin{equation}\label{eq:lambda}
u'(-\alpha)+\lambda u'(\alpha)=0.
\end{equation}
To obtain solutions of this overdetermined problem, we will always assume
\begin{equation}\label{eq:hp_a_alpha}
a > \frac{\pi^2}{4 \alpha^2},
\end{equation}
as otherwise \eqref{eq:u} does not admit solutions, for any $\omega$, see equation \eqref{eq:mu<0} 
ahead. To state our main result, it is convenient to introduce the thresholds
\begin{equation}\label{eq:omega*}
\omega^*:=\sqrt{4 a - \frac{\pi^2}{\alpha^2}},\qquad\qquad
\lambda^*:=e^{\omega^*\alpha}.
\end{equation}
Our main results are the following.

\begin{theorem}\label{thm:main1}
Let us assume \eqref{eq:HP_f_base}, \eqref{eq:HP_f_enhancedKPP}, \eqref{eq:hp_a_alpha} and \eqref{eq:omega*}.

For every choice of
\[
\frac{1}{\lambda^*}< \lambda < \lambda^*
\]
there exists a unique rotation speed $\omega =\omega(\lambda)$ such that the 
overdetermined problem \eqref{eq:u}, \eqref{eq:lambda} admits a positive solution  
$u=u(\cdot,\lambda)\in C^{3}([-\alpha,\alpha])$, which in turn is unique, too.

On the other hand, if
\[
\lambda\not\in\left(\frac{1}{\lambda^*}, \lambda^*\right)
\]
then the  overdetermined problem \eqref{eq:u}, \eqref{eq:lambda} has no positive 
solution, for any choice of $\omega$.
\end{theorem}
\begin{theorem}\label{thm:main2}
In the setting of Theorem \ref{thm:main1}, the map 
\[
\left(\frac{1}{\lambda^*}, \lambda^*\right) \ni \lambda\mapsto \left(u(\cdot,\lambda),\omega(\lambda)
\right) \in C^{2}([-\alpha,\alpha])\times (-\omega^*,\omega^*)
\]
is of class $C^1$ and satisfies
\[
\left(u(x,\lambda^{-1}),\omega(\lambda^{-1}) \right)= \left(u(-x,\lambda),-\omega(\lambda)
\right).
\]
Moreover (see Figure \ref{fig:la to om})
\[
\omega\text{ is strictly increasing and surjective,}
\qquad\omega(1)=0,\quad\text{and } \lim_{\lambda\to\lambda^*}\omega(\lambda)=\omega^*.
\]
\end{theorem}
\begin{figure}[t]
\centering
	\input{logistic_BW/log_lambda_omega.pgf}
\caption{numerical plot of the map $\lambda\mapsto\omega(\lambda)$, in logarithmic scale ($\alpha=
\pi$, $f(s)=s-s^2$).}
\label{fig:la to om}
\end{figure}
\begin{proof}[Proof of Theorem \ref{thm:coro}]
The theorem readily follows from Theorems \ref{thm:main1}, \ref{thm:main2}, taking 
$f(u)=a u - b u^p$, with $a,b>0$ and $p>1$, and $\alpha=\frac{\pi R}{k}$.
\end{proof}
\begin{remark}\label{rmk:more_regular}
We stress that the assumptions on $f$ readily guarantee that any solution $u$ of \eqref{eq:u} is of class 
$C^3([-\alpha,\alpha])$, but the map 
\[
\left(\frac{1}{\lambda^*}, \lambda^*\right) \ni \lambda\mapsto \left(u(\cdot,\lambda),\omega(\lambda)
\right) \in C^{3}([-\alpha,\alpha])\times (-\omega^*,\omega^*)
\]
is only continuous. The choice of the space $C^2([-\alpha,\alpha])$ in Theorem \ref{thm:main2} is made
to obtain that the dependence of $u$ on $\omega$ is of class  $C^1$. 
\end{remark}

\begin{remark}\label{rmk:analisi}
We find nontrivial solutions if and only if 
\[
|\log \lambda| < \alpha \sqrt{4 a - \frac{\pi^2}{\alpha^2}}.
\]
In particular, for any fixed domain and reaction, if the asymmetry parameter is too large then the only 
possible (equivariant) steady state/periodic solution is that corresponding to the total extinction. On the other 
hand, if the domain, and hence $\alpha$, is sufficiently large, then $\lambda$ and $\omega$ can be taken 
large, as well. In particular, for $\alpha\to+\infty$ we obtain $\omega^* \to 2\sqrt{f'(0)} = c_{KPP}$, 
which is the well known speed of invasion of a compactly supported initial datum for the global Cauchy 
problem with logistic reaction in $\R$ \cite{MR511740}.
\end{remark}

The paper is structured as follows. In Section \ref{sec:omega_fixed} we tackle problem \eqref{eq:u}, 
taking $\omega$ as a fixed parameter, and we show existence and uniqueness of the solution in the 
appropriate range, as well as some further relevant property. In this way, a map $\omega \mapsto 
\lambda(\omega)$ is naturally well-defined. Section \ref{sec:invertibilty} is devoted to prove that such 
map is monotone, and hence invertible. Finally, in Section \ref{sec:bif} we perform a bifurcation 
analysis of problem \eqref{eq:u}, which in particular yields the relation between $\omega^*$ and 
$\lambda^*$ in \eqref{eq:omega*}, together with some refined local information about the bifurcation diagram under 
different choices of the reaction $f$.

\section{Analysis for fixed \texorpdfstring{$\omega$}{omega}}\label{sec:omega_fixed}

In this section we consider problem \eqref{eq:u}, understanding $\omega$ as a fixed parameter. We are going to prove the following result.
\begin{proposition}\label{prop:main_both}
Let us assume \eqref{eq:HP_f_base}, \eqref{eq:HP_f_enhancedKPP}, \eqref{eq:hp_a_alpha} and \eqref{eq:omega*}. 
For every choice of
\[
-\omega^*< \omega < \omega^*
\]
there exists a unique solution  
$u=u(\cdot,\omega) \in C^{3 }([-\alpha,\alpha])$ of \eqref{eq:u}. On the other hand, if $|\omega|\ge\omega^*$, then \eqref{eq:u} has no positive solution.

The map 
\[
\left(-\omega^*,\omega^*\right) \ni \omega\mapsto u(\cdot,\omega) 
\in C^{2 }([-\alpha,\alpha])
\]
is of class $C^1$ and satisfies (see Figure \ref{fig:u})
\begin{equation}\label{eq:symm_u}
u(x,-\omega) = u(-x,\omega).
\end{equation}
\end{proposition}
\begin{figure}[t]
\centering
	\input{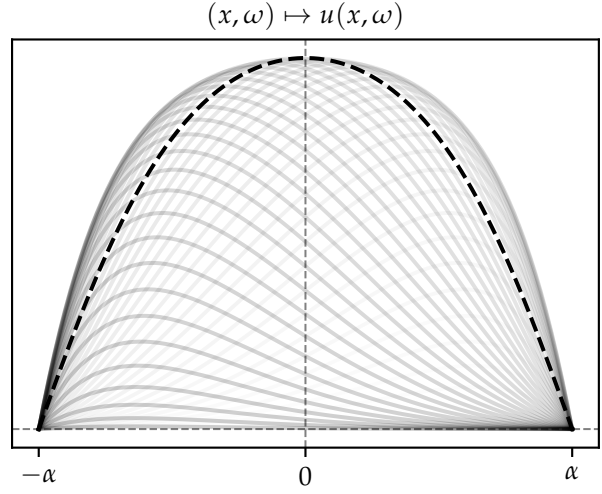}
\caption{numerical plot of the map $\omega\mapsto u(\cdot,\omega)$ ($\alpha=
\pi$, $f(s)=s-s^2$). The dashed line corresponds to the symmetric case $\omega=0$.}
\label{fig:u}
\end{figure}
\begin{remark}\label{rmk:symm_u}
In particular, \eqref{eq:symm_u} implies the well-known symmetry properties of the stationary solution 
of the Fisher-KPP equation without drift term, which is recovered with $\omega=0$:
\[
u(-x,0)=u(x,0).
\]
\end{remark}

The rest of this section is devoted to the proof of the proposition above.\medskip

First, the existence of a positive solution of \eqref{eq:u}, with $\omega$ in the appropriate range, 
follows from a direct application of the method of sub/supersolutions. Although part of the argument is quite standard, we summarize it both for the sake of completeness and because some details in the 
calculations will be used in the following.

The key point is that the existence of solutions is related to the sign of the principal eigenvalue 
of the linearized problem:
\begin{equation*}%\label{eq:linearized}
\begin{cases}
-\varphi_{\omega}'' - \omega \varphi_{\omega}' - a \varphi_{\omega} = \mu_{\omega} 
\varphi_{\omega}& \text{in } (-\alpha,\alpha)\\
        \varphi_{\omega} > 0 & \text{in } (-\alpha,\alpha)\\
        \varphi_{\omega}(-\alpha)=\varphi_{\omega}(\alpha)=0.
\end{cases}
\end{equation*}
By direct computations, in this case the eigenvalue and eigenfunction are explicit (the latter, 
up to a normalizing constant):
\[
\mu_{\omega}= \frac{\pi^2}{4 \alpha^2} + \frac{\omega^2}{4} - a,\qquad
\varphi_{\omega}(x)= e^{-\omega x /2} \cos{\left(\frac{\pi}{2\alpha}x\right)}.
\]
In particular, by \eqref{eq:hp_a_alpha}, \eqref{eq:omega*}, we have that
\begin{equation}\label{eq:mu<0}
\mu_{\omega} <0 \qquad\iff\qquad \text{\eqref{eq:hp_a_alpha} holds true and }|\omega|<\omega^*.
\end{equation}
This immediately rules out the existence of positive solutions for $|\omega|\ge\omega^*$ (or when \eqref{eq:hp_a_alpha} fails): indeed, would such a solution exist, one could test \eqref{eq:u} with $\varphi_{-\omega}$, obtaining 
\[
0>-\int_{-\alpha}^\alpha g(u)\varphi_{-\omega} = \mu_{-\omega} \int_{-\alpha}^\alpha u \varphi_{-\omega}
\ge0\qquad 
\text{when }|\omega|\ge\omega^*,
\]
contradiction (recall \eqref{eq:f_g}).

To find positive solutions of \eqref{eq:u}, 
we proceed with a series of lemmas.
\begin{lemma}\label{lem:apriori+subsup}
Assuming $|\omega|<\omega^*$ and recalling that $\bar u$ is defined in \eqref{eq:HP_f_base}, we have:
\begin{enumerate}
\item\label{i:1} any positive solution $u$ of \eqref{eq:u} admits a unique maximum point 
$x_m\in(-\alpha,\alpha)$, 
with 
\[
u'>0\text{ on }[-\alpha,x_m),\qquad 
u'<0\text{ on }(x_m,\alpha],\qquad 
u(x_m) < \bar u;
\]
\item\label{i:2} $u^*(x)\equiv \bar u$ is a supersolution of \eqref{eq:u};
\item\label{i:3} there exists $\eps>0$ such that $u_* := \eps \varphi_{\omega}$ is a subsolution of \eqref{eq:u} 
and $u_*\le u^*$ in $[-\alpha,\alpha]$.
\end{enumerate}
\end{lemma}
\begin{proof}
Denote with $u$ a positive solution of \eqref{eq:u} and with 
$x_m\in(-\alpha,\alpha)$ a maximum point of $u$. By uniqueness of the solution of the Cauchy 
problem for ODEs, one can easily infer $u(x_m) \neq \bar u$. Since 
\[
0\le -u''(x_m) - \omega u'(x_m) = f(u(x_m)),
\end{equation*}
from assumption \eqref{eq:HP_f_base} we deduce that  $u(x_m) < \bar u$; then, again by \eqref{eq:HP_f_base}, 
any other possible critical point of $u$ is a strict local maximum, and \ref{i:1} follows.

Next, \ref{i:2} follows by definition. Finally, to prove \ref{i:3}, we observe that 
$u_* := \eps \varphi_{\omega}$ is a subsolution if and only if
\begin{equation*}
 (\mu_\omega+a) u_* - f(u_*)= \mu_\omega u_* + g(u_*)\le 0\qquad
 \text{in }(-\alpha,\alpha).
\end{equation*}
Now, since $\mu_\omega<0$ by \eqref{eq:mu<0}, \eqref{eq:f_g} implies the existence of 
$s_0>0$ such that  
\[
g(s) \le -\frac{\mu_\omega}{2}s\qquad\text{ for }0\le s \le s_0.
\]
Then, to conclude the proof it is enough to take
\[
0<\eps \le \frac{\min(s_0,\bar u)}{\|\varphi_{\omega}\|_\infty}.\qedhere
\]
\end{proof}
Next, motivated by the first point of the previous lemma, we define 
\[
M:=\max_{[0,\bar u]} -f' = -f'(\bar u) 
\]
and
\[
T\from \{v\in C([-\alpha,\alpha]):0\le v \le \bar u \} \to C^2_0([-\alpha,\alpha])=:
\{v\in C^2([-\alpha,\alpha]):v(\pm\alpha)=0 \},
\]
in such a way that $w=T(v)$ if and only if 
\begin{equation*}%\label{eq:fixedpoint}
        \begin{cases}
            -w''-\omega w' + M w  = f(v) + Mv =: h(v)  & \text{in $(-\alpha,\alpha)$}\\
            w(-\alpha)=w(\alpha)=0.
        \end{cases}
\end{equation*}
Using e.g.~Lax-Milgram theorem and regularity theory (or representation with the Green function), it is not difficult to check that $T$ is well-defined, continuous, and that there exists a constant $C>0$ such that
\begin{equation}\label{eq:reg_T}
v\in C([-\alpha,\alpha]),\quad 0\le v \le \bar u 
\qquad\implies\qquad
\begin{cases}
\|T(v)\|_{C^{2}} \le C \smallskip\\
\|T(v)\|_{C^{3}} \le C \|v\|_{C^{1}} \text{ if }v\in C^1\text{, too.} 
\end{cases}
\end{equation}
Moreover, $h'(s)=f'(s)+M \ge 0$, i.e.~$h$ is increasing in $s$, and $u$ solves 
\eqref{eq:u} if and only if it is a fixed point of $T$.

\begin{lemma}\label{L1}
For any $v,w\in C([-\alpha,\alpha])$ with $0\le v,w \le \bar u$. 
Then:
\begin{enumerate}
\item if $v \le w$ then $T(v) \le T(w)$;
\item if $v$ is a subsolution of \eqref{eq:u} then $T(v) \ge v$ and $T(v)$ is a subsolution;
\item if $w$ is a supersolution of \eqref{eq:u} then $T(w) \le w$ and $T(w)$ is a supersolution.
\end{enumerate}
\end{lemma}
\begin{proof}
All these facts are direct consequences of the already mentioned monotonicity properties of $h$, 
and of the fact that the operator
\[
\varphi \mapsto  - \varphi'' -\omega \varphi' + M \varphi 
\]
with Dirichlet boundary conditions in $(-\alpha,\alpha)$ enjoys the weak maximum principle.
\end{proof}

\begin{proof}[Proof of Proposition \ref{prop:main_both} -- Existence]
Assuming $|\omega|<\omega^*$ and recalling the definitions of $u_*$, $u^*$ from Lemma 
\ref{lem:apriori+subsup}, we define two sequences:
 \begin{equation*}
        \begin{cases}
            v_0 = u_*\\
            v_{n+1}=T(v_n),
        \end{cases}
\qquad
        \begin{cases}
            w_0 = u^*\\
            w_{n+1}=T(w_n).
        \end{cases}
 \end{equation*}
By induction, using Lemma \ref{L1}, we have that  $(v_n)_n$ is a (pointwise) increasing sequence of 
subsolutions, $(w_n)_n$ a decreasing sequence of supersolutions and, for every $n$,
\begin{equation*}%\label{order_v_w}
0 \le v_{n} \le w_{n} \le \bar{u}.
\end{equation*}
Since $v_n$ is increasing and bounded, there exists a limit $\hat{u}$ such that 
$v_n \to \hat{u}$ pointwise. Recalling \eqref{eq:reg_T}, we obtain that 
$(v_n)_n = (T(v_{n-1}))_n$ is uniformly bounded in $C^{2}$, and then in $C^{3}$. By Ascoli-Arzelà 
theorem, from any subsequence we can extract a further subsequence converging to $\hat u$ in 
$C^{2}$. Finally, since this is true for any subsequence, we have that the full sequence converges 
to $\hat u$ in $C^{2}$. By continuity of $T$, we can pass to the limit in the relation $v_{n+1}=T(v_n)$, 
obtaining that $\hat{u} \in C^{2}_0([-\alpha,\alpha])$ satisfies $T(\hat{u})=\hat{u}$, thus $\hat{u}$ is a solution of \eqref{eq:u} (and actually $\hat{u}$ is of class $C^3$). 
\end{proof}

\begin{proof}[Proof of Proposition \ref{prop:main_both} -- Uniqueness]
Take $u$, $v$ positive solutions of \eqref{eq:u}. Notice that, by Hopf lemma 
(or simply by uniqueness of the solution of the Cauchy problem), we have that 
$su < v$ in $(-\alpha,\alpha)$ for $s>0$ sufficiently small, and $su > v$ in $(-\alpha,\alpha)$ for 
$s$ sufficiently large. Let us define
\[
\bar s = \sup \{s>0: su < v \text{  in }  (-\alpha,\alpha) \}\in(0,\infty);
\] 
then $\bar s u$ touches $v$ from below, in the sense that $su \le v$, and there exists 
$\bar x \in [-\alpha,\alpha]$ such that
\[
\bar s u (\bar x) = v (\bar x),\qquad \bar s u' (\bar x) = v' (\bar x).
\]
Let us assume by contradiction  that $0<\bar s < 1$. Then \eqref{eq:HP_f_monoto} implies that 
$w = \bar s u$ is a strict subsolution of \eqref{eq:u}, since in such case
\[
 -w'' -\omega w' = \bar s f(u) = \bar s u\cdot \frac{f(u)}{u} < \bar s u\cdot \frac{f(\bar s u)}{\bar s u} = f(w);
\]
since  $\bar s u$ touches $v$ from below, we obtain a contradiction with the strong maximum principle. 

Summarizing, we have proved that $\bar s\ge1$, that is $u \ge v$ in $(-\alpha,\alpha)$; exchanging the 
role of $u$ and $v$ we conclude.  
\end{proof}

The existence and uniqueness of the solution of \eqref{eq:u} that we just proved allows to define the 
map 
\[
\left(-\omega^*,\omega^*\right) \ni \omega\mapsto u(\cdot,\omega)
\]
Moreover, by uniqueness and direct computations, it immediately follows 
\[
u(x,-\omega) = u(-x,\omega).
\]
We are left to prove the regularity of the map  $\omega\mapsto u(\cdot,\omega)$. 
This will follow by applying the implicit function theorem to the map:
\begin{equation*}
    F : C_0^{2}([-\alpha,\alpha]) \times \mathbb{R} \to C([-\alpha,\alpha]),\end{equation*}
\begin{equation*}
F(\omega,u)=
      u''  +\omega u' + f(u)
\end{equation*}
Of course, $F$ is of class $C^1$ by the assumptions on $f$ (notice that this would not be the case taking as domain $C^3_0$ and target $C^1$, as $f$ is only $C^1$); to conclude, we need to check that
$\partial_{u} F(\omega,u)$ is invertible between the appropriate spaces, where the pair 
$(\omega,u)$ is the unique solution of \eqref{eq:u} for $\omega$ fixed.

Since 
\begin{equation*}
    \partial_{u} F(\omega,u)[v] = 
        v'' + \omega v' +  f'(u)v,
        \qquad \text{for every }v\in C_0^{2}([-\alpha,\alpha]),
\end{equation*}
by the Fredholm alternative we have that the invertibility of  $\partial_{u} F(\omega,u)$ is equivalent 
to its injectivity. Then the end of the proof of Proposition \ref{prop:main_both} will follow from 
the next lemma.
\begin{lemma}\label{lem:lin_injective}
The problem 
\begin{equation}\label{eq:lin_inj}
    \begin{cases}%[left=\empheqlbrace]
        -w'' -\omega w' -f'(u)w = 0 & \text{in }(-\alpha,\alpha) \\
        w(-\alpha) = w(\alpha) =0
    \end{cases}
\end{equation}
only admits the trivial solution $w\equiv 0$.
\end{lemma}
\begin{proof}
Assume that $w$ is a nontrivial solution of \eqref{eq:lin_inj}; to obtain a contradiction, we apply a 
Sturm-Picone argument. By uniqueness of the solution of the Cauchy problem we have that $w$ has only 
simple zeroes. As a consequence, up to a change of sign, we can assume without loss of generality 
that, for some $-\alpha<\bar x \le \alpha$,
\[
w>0\text{ in }(-\alpha,\bar x), \qquad w(\bar x) = 0 ,\qquad w'(\bar x) < 0.
\] 
Testing \eqref{eq:lin_inj} with $e^{\omega x} u$ and \eqref{eq:u}  with $e^{\omega x} w$ on 
$(-\alpha,\bar x)$ we obtain
\[
\int_{-\alpha}^{\bar x} e^{\omega x}f'(u)uw = - e^{\omega \bar x} u(\bar x)w'(\bar x) +\int_{-\alpha}^{\bar x} e^{\omega x}u'w' \ge 
\int_{-\alpha}^{\bar x} e^{\omega x}f(u)w, 
\]
whence, by assumption \eqref{eq:HP_g_standardKPP},
\begin{equation*}
0\ge  \int_{-\alpha}^{\bar x} e^{\omega x} [ g'(u)u - g(u) ] w> 0 ,
\end{equation*}
a contradiction.
\end{proof}
As a byproduct of the last lemma, we have that the linearized operator satisfies both the weak and the strong maximum principle.
\begin{corollary}\label{coro:mp}
Let $u$ as above, and assume that $-\alpha\le x_1 < x_2 \le \alpha$ and that $w$ satisfies
\begin{equation*}%\label{eq:lin_inj}
    \begin{cases}%[left=\empheqlbrace]
        -w'' -\omega w' -f'(u)w \ge 0 & \text{in }(x_1,x_2) \\
        w(x_1)\ge0, \qquad w(x_2) \ge0.
    \end{cases}
\end{equation*}
Then
\[
\text{either }\qquad w\equiv0\qquad \text{ or }\qquad w>0\qquad\text{ in }(x_1,x_2),
\]
and, in the latter case, $w(x_i)=0$ implies $w'(x_i)\neq0$, $i=1,2$.
\end{corollary}
\begin{proof}
By the Hopf lemma and the strong maximum principle, it is well-known that a necessary and sufficient condition for the claim is that the principal 
eigenvalue $\Lambda$ of the operator
\[
w \mapsto -w'' -\omega w' -f'(u)w,
\]
with Dirichlet boundary conditions on $(-\alpha,\alpha)$, is strictly positive, see e.g.~\cite{MR1258192}. On the one hand, we have: 
\begin{equation*}
\Lambda = \sup_{\phi>0} \inf_{(-\alpha,\alpha)} \frac{- \phi'' -\omega \phi' -f'(u) \phi}{\phi} \ge \inf_{(-\alpha,\alpha)} \frac{- u'' -\omega u' -f'(u) u}{u} =  \inf_{(-\alpha,\alpha)} \frac{g'(u)u - g(u)}{u} =0
\end{equation*}
(recall that $g$ is convex). On the other hand, Lemma \ref{lem:lin_injective} implies that 
$\Lambda \neq 0$, and the corollary follows.
\end{proof}

\section{Global inversion}\label{sec:invertibilty}

In view of Proposition \ref{prop:main_both}, throughout this section we denote by $u=u(\cdot,\omega) \in 
C^{2}_0([-\alpha,\alpha])$ the unique positive solution of \eqref{eq:u}, where the dependence on 
$\omega$, $|\omega|<\omega^*$, is of class $C^1$. Accordingly, the map
\begin{equation}\label{eq:lambda(omega)}
\lambda(\omega) := -\frac{u'(-\alpha,\omega)}{u'(\alpha,\omega)}>0
\end{equation}
is well-defined, by Lemma \ref{lem:apriori+subsup}, and of class $C^1$ too. With this notation, we have that $u(\cdot,\omega)$ solves the overdetermined 
problem \eqref{eq:u}, \eqref{eq:lambda}, with asymmetry parameter $\lambda=\lambda(\omega)$. 
Consequently, the proof of our main results boils down to the invertibility of the map 
$\omega\mapsto\lambda(\omega)$. In the following, when no confusion arises, we will systematically 
omit the explicit dependence on $\omega$, which is always understood. 

By differentiating, we obtain that
\begin{equation}\label{eq:def_v}
v:=\frac{\partial}{\partial \omega} u \in C^{2}([-\alpha,\alpha])
\quad\text{satisfies }
\qquad
\begin{cases}
-v'' - \omega v' -f'(u)v=u'\\
v(-\alpha) = v(\alpha) = 0,
\end{cases}
\end{equation}
while
\begin{equation}\label{eq:lambda_dot}
\dot \lambda :=\frac{d}{d \omega} \lambda
\quad\text{satisfies }
\qquad
v'(-\alpha) + \lambda v'(\alpha) + \dot \lambda u'(\alpha) = 0.
\end{equation}
Moreover,  \eqref{eq:symm_u} yields
\begin{equation}\label{eq:symm_lambda}
\lambda(-\omega) := -\frac{u'(-\alpha,-\omega)}{u'(\alpha,-\omega)} = 
-\frac{-u'(\alpha,\omega)}{-u'(-\alpha,\omega)} = \frac{1}{\lambda(\omega)},\qquad\qquad
\dot\lambda(-\omega)=\frac{\dot\lambda(\omega)}{\lambda(\omega)^2}.
\end{equation}

The main result of this section is the following.
\begin{proposition}\label{prop:lambda_dot>0}
For every $-\omega^*<\omega<\omega^*$, we have $\dot\lambda(\omega)>0$ (see Figure \ref{fig:om to la}).
\end{proposition}
\begin{figure}
\centering
	\input{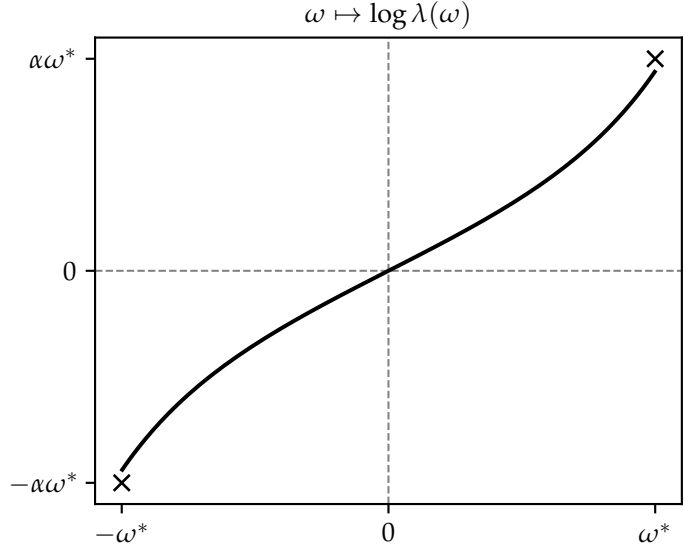}
\caption{numerical plot of the map $\omega\mapsto\lambda(\omega)$, in logarithmic scale ($\alpha=
\pi$, $f(s)=s-s^2$).}
\label{fig:om to la}
\end{figure}
\begin{remark}\label{rmk:after_prop2}
By the inversion theorem, the previous proposition implies that, for a suitable choice of the 
interval $I$, we can define the map 
\[
I\ni \lambda \mapsto \omega(\lambda) \in (-\omega^*,\omega^*),
\] 
which is $C^1$, strictly increasing and onto. This, together with Proposition \ref{prop:main_both}, will  
provide almost all our main results. The only missing part will be the identification of the interval
\[
I=\left(\lim_{\omega\to-\omega^*}\lambda(\omega),\lim_{\omega\to\omega^*}\lambda(\omega)\right)
\]
which is discussed in Section \ref{sec:bif}.  
\end{remark}

The rest of this section is devoted to the proof of Proposition \ref{prop:lambda_dot>0}. 
By \eqref{eq:lambda_dot}, and since $u'(\alpha)<0< \lambda$, for every relevant $\omega$, the proof of 
Proposition \ref{prop:lambda_dot>0} requires a refined study of the properties of $v$, and in particular of the 
sign of $v'$ at $\pm\alpha$. Due to \eqref{eq:symm_lambda}, we restrict our study to the case
\[
0\le \omega <\omega^*,
\]
as the proof in such interval will imply also that for $\omega<0$. 
As we will see, we have to take into account two possible regimes for $v$, depending on $\omega$. 

To start with, we draw some consequences of the maximum principle on $v$. To this aim, 
we recall from conclusion \ref{i:1} of Lemma \ref{lem:apriori+subsup} the sign properties of $u'$ (i.e.~the right hand side of \eqref{eq:def_v}): denoting with $x_m=x_m(\omega)\in(-\alpha,\alpha)$ the unique 
maximum point of $u$, we know that $u'>0$ in $[-\alpha,x_m)$ and $u'<0$ in $(x_m,\alpha]$. 
We obtain the following.
\begin{lemma}\label{lem:cases_v}
Assume that $0\le \omega <\omega^*$. Then one of the following alternatives holds true:
\begin{enumerate}
\item\label{i:v1} either there exists a unique $\bar x \in (-\alpha,\alpha)$ such that
\begin{equation*}
\begin{cases}
v > 0 & \text{in } (-\alpha,\bar x)\\
v < 0 & \text{in } (\bar x,\alpha),
\end{cases}
\qquad\text{ and }\quad  v'(\pm\alpha)>0;  
\end{equation*}
\item\label{i:v2} or $v<0$ in $(-\alpha,\alpha)$ and $v'(-\alpha)\le 0< v'(\alpha)$.
\end{enumerate}
\end{lemma}
\begin{proof}
Let $x_m\in(-\alpha,\alpha)$ denote the unique maximum point of $u$. We have three possibilities.

If $v(x_m)=0$, then we can apply to \eqref{eq:def_v} the maximum principle Corollary 
\ref{coro:mp}, in both $(-\alpha,x_m)$ and $(x_m,\alpha)$, obtaining that alternative \ref{i:v1} 
is verified, with $\bar x = x_m$.

In case $v(x_m)<0$, then Corollary \ref{coro:mp} applies in $(x_m,\alpha)$, yielding $v<0$ in such 
interval. Let us define $\bar x := \inf\{-\alpha<x<x_m:v<0\text{ in }(x,x_m]\}$. Then either $\bar x = 
-\alpha$, and alternative \ref{i:v2} is verified; or $\bar x >-\alpha$, in which case $v(\bar x)=0$, 
$v<0$ in $(\bar x, \alpha)$, and 
Corollary \ref{coro:mp} yields $v>0$ in $(-\alpha,\bar x)$. Then alternative \ref{i:v1} is verified, with 
the same $\bar x$.

Finally, let us assume $v(x_m)>0$. Arguing as in the previous case, and defining now $\bar x := 
\sup\{x_m<x<\alpha:v>0\text{ in }[x_m,x)\}$, we have that $\bar x<\alpha$ provides once again  
alternative \ref{i:v1}, and we are left to exclude that $\bar x=\alpha$ (which would imply $v>0$ in $(-\alpha,\alpha)$). Assume by contradiction 
that $v>0$ in $(-\alpha,\alpha)$. We test the equation for $u$ with $e^{\omega x}v$ and that for 
$v$ with $e^{\omega x}u$, obtaining
\[
\int_{-\alpha}^\alpha e^{\omega x} f(u) v = \int_{-\alpha}^\alpha e^{\omega x} (-u'' - \omega u')v =
\int_{-\alpha}^\alpha e^{\omega x} u'v',
\]
\[
\begin{split}
\int_{-\alpha}^\alpha e^{\omega x} (f'(u)v+u')u = \int_{-\alpha}^\alpha e^{\omega x} (-v'' - \omega v')u =
\int_{-\alpha}^\alpha e^{\omega x} u'v' .
\end{split}
\]
Subtracting, and exploiting assumption \eqref{eq:HP_g_standardKPP}, we infer
\begin{equation}\label{eq:u'byparts}
0<\int_{-\alpha}^\alpha e^{\omega x} [ g'(u)u - g(u) ] v = \int_{-\alpha}^\alpha e^{\omega x} u' u = - \int_{-\alpha}^\alpha \frac{\omega}{2} e^{\omega x} u^2 \le0,
\end{equation}
a contradiction. (A similar argument can be used to show that alternative \ref{i:v2} cannot happen when $\omega=0$, but this can be obtained directly, see Remark \ref{rmk:cases_v} ahead.)
\end{proof}
\begin{remark}\label{rmk:changing_sign_rhs}
With a closer look to the previous proof, we see that whenever some $w$ satisfies
\begin{equation}\label{eq:new_w_v}
\begin{cases}
-w'' - \omega w' -f'(u)w=r\\
w(-\alpha) = w(\alpha) = 0,
\end{cases}
\qquad\text{with}\qquad
\begin{cases}
r > 0 & \text{in } (-\alpha,x_0)\\
r < 0 & \text{in } (x_0,\alpha),
\end{cases}
\end{equation}
for some $x_0\in(-\alpha,\alpha)$, then either $w$ has no zero in $(-\alpha,\alpha)$, or 
there exists a unique $\bar x \in (-\alpha,\alpha)$ such that
\begin{equation*}
\begin{cases}
w > 0 & \text{in } (-\alpha,\bar x)\\
w < 0 & \text{in } (\bar x,\alpha),
\end{cases}
\qquad\text{ and }\quad  w'(\pm\alpha)>0.  
\end{equation*}
Indeed, it is clear that \eqref{eq:new_w_v} holds true with the choice $w=v$, $r=u'$ and $x_0=x_m$, and 
the proof of Lemma \ref{lem:cases_v} is based only on the sign properties of $u'$, except for the integration by parts in \eqref{eq:u'byparts}. Thus, repeating such proof, either $w$ satisfies one of the alternatives for $v$, or $w>0$ in  $(-\alpha,\alpha)$.
\end{remark}
\begin{remark}\label{rmk:cases_v}
It is worth noticing that both alternatives in Lemma \ref{lem:cases_v} are verified, for some values of 
$\omega$. Indeed, by the very same lemma, we have that if $\omega=0$ then 
alternative \ref{i:v1} is verified, with $\bar x=x_m=0$. This can be easily checked: by differentiating 
\eqref{eq:symm_u} we obtain
\[
v(x,-\omega)=-v(-x,\omega),\qquad\text{ whence }\qquad
v(-x,0)=-v(x,0),
\] 
which implies $v(0,0)=0$, and the claim follows by the lemma. On the other hand, we are going to see that 
also alternative \ref{i:v2} is verified for some $\omega$, in particular when $\omega$ is near 
$\omega^*$, see Remark \ref{rmk:altern_2} ahead. As a consequence, by continuity, there exists at least 
a positive value of $\omega$ such that alternative \ref{i:v2} is verified and $v'(-\alpha)=0$. Notice that 
this is not in contradiction with the Hopf principle contained in Corollary \ref{coro:mp}. 
Actually, numerical simulations show that case \ref{i:v1} happens for 
$-\bar \omega < \omega < \bar \omega$, for a suitable $\bar \omega <\omega^*$; 
accordingly, case \ref{i:v2} is true when $\bar\omega\le \omega < \omega^*$, with $v'(-\alpha,\bar\omega)=0$, 
while $-\omega^*<\omega\le-\bar \omega$ yields $v<0$ in $(-\alpha,\alpha)$  (see Figure 
\ref{fig:v}).
\end{remark}
\begin{figure}[t]
\centering
	\input{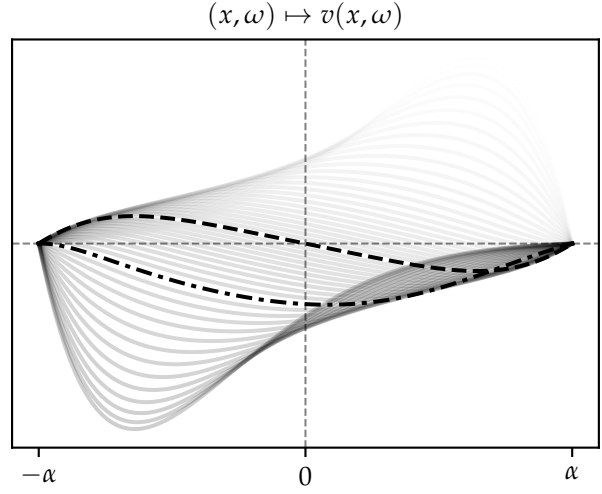}
\caption{numerical plot of the map $\omega\mapsto v(\cdot,\omega)$ ($\alpha=
\pi$, $f(s)=s-s^2$). The dashed line corresponds to the symmetric case $\omega=0$, 
while the dash-dotted one to a case in which $\omega=\bar\omega$ yields $v'(-\alpha,\bar\omega)=0$.}
\label{fig:v}
\end{figure}

Now, if $v$ satisfies alternative \ref{i:v1} in Lemma \ref{lem:cases_v} for some $\omega$, e.g.~if 
$\omega=0$, then \eqref{eq:lambda_dot} readily implies that $\dot\lambda(\omega)>0$. On the other 
hand, when $v < 0$ in $(-\alpha,\alpha)$, the study of the sign of $\dot \lambda$ is more delicate. 
To deal with it, we will apply Remark \ref{rmk:changing_sign_rhs} to a modified problem, involving a 
normalization of $u$ with respect to its maximum value $u(x_m)$.

\begin{lemma}\label{lem:tilde}
For $|\omega|<\omega^*$ let us define
\[
\tilde u(\cdot,\omega) := \frac{1}{u(x_m(\omega))}\,u(\cdot,\omega).
\]
Then the map $\omega\mapsto\tilde u$ is of class $C^1$, and
\begin{equation}\label{eq:v_tilde}
\tilde v := \frac{\partial}{\partial \omega} \tilde u=\frac{1}{u(x_m)}\,v-\frac{v(x_m)}{u(x_m)^2}\,u
\quad\text{ satisfies }\qquad
\begin{cases}
-\tilde v '' - \omega \tilde v' - f'(u) \tilde v = r & \text{in } (-\alpha,\alpha)\\
\tilde v(-\alpha)=\tilde v(\alpha)=0,
\end{cases}
\end{equation}
where
\begin{equation}\label{eq:r,h}
r:=\frac{g'(u)u-g(u)}{u(x_m)} \left [h - \frac{v(x_m)}{u(x_m)} \right ],\qquad\qquad
h:=\frac{u'}{g'(u)u-g(u)}.
\end{equation}
\end{lemma}
\begin{proof}
First, we notice that  the map $x_m\from(-\omega^*,\omega^*) \to (-\alpha,\alpha)$ 
is implicitly defined by the equation $u'(x_m,\omega)=0$. By point \ref{i:1} of Lemma 
\ref{lem:apriori+subsup} we have that $u(x_m)<\bar u$, thus assumption \eqref{eq:HP_f_base} 
and the equation for $u$ yield
\[
\partial_{x} u'(x_m,\omega) = u''(x_m,\omega) = f(u(x_m,\omega)) <0;
\]
 the implicit function theorem applies, and $x_m$ is a $C^1$ function of $\omega$. 
As a straightforward consequence, we obtain that  $K(\omega):=u(x_m(\omega),\omega)$ is $C^1$ too, with
\[
\dot K(\omega)=\frac{d}{d\omega}\,u(x_m(\omega),\omega)=u'(x_m(\omega),\omega)\cdot \dot x_m(\omega) + v(x_m(\omega),\omega)=v(x_m).
\]
We obtain
\[
\tilde v(\cdot,\omega)= \frac{\partial}{\partial\omega}\,\frac{u(\cdot,\omega)}{K(\omega)}=
\frac{v(\cdot,\omega)}{K(\omega)}-\frac{\dot K(\omega)}{K(\omega)^2}\,u(\cdot,\omega)=\frac{1}{u(x_m)}\,v-\frac{v(x_m)}{u(x_m)^2}\,u.
\]
Since the equation for $u$ can be written as 
\[
-u''-\omega u' - f'(u)u=f(u)-f'(u)u=-g(u)+g'(u)u,
\] 
by linearity we have that $\tilde v$ satisfies the problem in  \eqref{eq:v_tilde}, with right hand side
\[
r=\frac{1}{u(x_m)}\,u'-\frac{v(x_m)}{u(x_m)^2}\,\left(g'(u)u-g(u)\right)=
\frac{g'(u)u-g(u)}{u(x_m)} \left [\frac{u'}{g'(u)u-g(u)} - \frac{v(x_m)}{u(x_m)} \right ],
\]
which is \eqref{eq:r,h}.
\end{proof}
In view of Remark \ref{rmk:changing_sign_rhs}, the next step consists in studying the sign of $r$ on 
$(-\alpha,\alpha)$. To this aim we need a preliminary lemma, concerning the log-concavity of $u$.
\begin{lemma}\label{lem:z'}
We have  $\left(\dfrac{u'}{u}\right)'<0$ in $(-\alpha,\alpha)$.
\end{lemma}
\begin{proof}
Let us write $z:= \dfrac{u'}{u}$ and compute the derivative of $z$:
\begin{equation}\label{z_first_der}
\begin{split}
z'=\left ( \frac{u'}{u} \right )' &= \frac{u''}{u} - \left ( \frac{u'}{u} \right )^2= \frac{-\omega u' - f(u)}{u} - \left ( \frac{u'}{u} \right )^2\\
&= - z^2 - \omega z - a + \frac{g(u)}{u},
\end{split}
\end{equation}
where we used that $u$ solves \eqref{eq:u}. We distinguish two cases.

If $-\alpha<x\le x_m$, then $z(x)\ge0$. Moreover, $0<u(x)<\bar u$ implies $f(u(x))=au(x)-g(u(x))>0$ by 
\eqref{eq:HP_f_base}; then \eqref{z_first_der} implies $z'(x) < 0$, for every $x \in (-\alpha,x_m]$.

On the other hand, if $x_m<x<\alpha$ then $z(x)<0$ and \eqref{z_first_der} does not provide enough 
information. Let us define $x_1:=\sup\{x_m<x<\alpha:z'<0\text{ in }(-\alpha,x)\}$ and assume by contradiction that $x_m<x_1<\alpha$. We obtain
\[
z(x_1)<0,\qquad z'(x_1)=0\qquad\text{ and }z''(x_1)\ge0,
\]
indeed, since $u$ is of class $C^3$ (recall Remark \ref{rmk:more_regular}), $z''$ is well defined, and $z'$ is not locally decreasing at $x_1$.

However, if we compute the second derivative of $z$, we have:
\begin{equation*}
\begin{split}
z'' &=- 2 z z' - \omega z' + \left( \frac{g(u)}{u} \right)' = - 2 z z' - \omega z' + \frac{g'(u)u'}{u} - \frac{u' g(u)}{u^2}\\
&= - 2 z z' - \omega z' + \frac{g'(u)u-g(u)}{u} \cdot \frac{u'}{u} = - 2 z z' - \omega z' + \frac{g'(u)u-g(u)}{u} \cdot z,
\end{split}
\end{equation*}
and, in particular,

\begin{equation*}
z''(x_1) = \left. \frac{g'(u)u-g(u)}{u} \right |_{x=x_1} \cdot z(x_1) < 0
\end{equation*}
by assumption \eqref{eq:HP_g_standardKPP}. But this is a contradiction, thus $x_1=\alpha$, and 
$z$ is strictly decreasing in $(-\alpha,\alpha)$.
\end{proof}

\begin{lemma}\label{lem:sign_h_r}
Let $r$ be defined as in \eqref{eq:r,h}, and assume that, for some fixed $\omega$, $v(x_m)<0$. Then
there exists $x_0 \in (x_m,\alpha)$ such that:
\begin{equation*}
\begin{cases}
r(x) > 0 & \text{in } (-\alpha,x_0),\\
r(x) < 0 & \text{in } (x_0,\alpha).
\end{cases}
\end{equation*}
\end{lemma}
\begin{proof}
Notice that, by \eqref{eq:r,h}, we can write $r$ in $(-\alpha,\alpha)$ as
\[
r(x)=c(x)[h(x)+k], \qquad \text{where }h(x):=\frac{u'(x)}{g'(u(x))u(x)-g(u(x))}\ ,
\quad c(x)>0,\quad k>0,
\]
by \eqref{eq:HP_g_standardKPP} and $v(x_m)<0$, respectively, while $h$ has the same sign as $u'$. Then, 
it is clear that $r>0$ on $(-\alpha,x_m]$, and that $h(x_m)=0$, $\lim_{x\to\alpha} h(x)=-\infty$. 
We claim that $h$ is strictly decreasing in $[x_m,\alpha)$: in such a case, we infer that the equation 
$h(x)=-k$ has exactly one solution $x_0\in (x_m,\alpha)$, and the lemma follows.

To prove the claim, we first notice that, by assumption \eqref{eq:HP_g_enhancedKPP}, the map
\[
x\mapsto g'(u(x))-\dfrac{g(u(x))}{u(x)}\qquad\text{ is strictly decreasing in }[x_m,\alpha),
\]
as $u$ is strictly decreasing in such interval. Then we can rewrite $h$ on $[x_m,\alpha)$ as 
\begin{equation*}
h=\frac{u}{g'(u)u-g(u)} \cdot \frac{u'}{u}
\end{equation*}
which is the product of a positive, strictly increasing factor and a negative strictly decreasing one, 
by Lemma \ref{lem:z'}. Then the claim follows, proving the lemma.   
\end{proof}
As we mentioned, the sign properties of $r$ have immediate consequences on those of $\tilde v$.
\begin{lemma}\label{lem:cases_vtilde}
Assume that $0\le \omega <\omega^*$, and that alternative \ref{i:v2} in Lemma \ref{lem:cases_v}
holds true. Then 
\begin{equation*}
\begin{cases}
\tilde v > 0 & \text{in } (-\alpha,x_m)\\
\tilde v < 0 & \text{in } (x_m,\alpha),
\end{cases}
\qquad\text{ and }\quad  \tilde v'(\pm\alpha)>0;  
\end{equation*}
\end{lemma}
\begin{proof}
By Lemmas \ref{lem:tilde}, \ref{lem:sign_h_r} and Remark \ref{rmk:changing_sign_rhs} we 
have that either $\tilde v$ has a definite sign in $(-\alpha,\alpha)$, or it has exactly one 
change of sign, from positive to negative. As a consequence, to prove the lemma it is enough to 
show that $\tilde v(x_m)=0$. But this is a direct consequence of the definition of $\tilde u$, $\tilde v$ and 
$x_m$, since it implies
\[
0=\frac{d}{d \omega}\,1 = \frac{d}{d \omega}\, \tilde u (x_m(\omega),\omega) = 
\tilde u' (x_m(\omega),\omega)\cdot \dot x_m(\omega) + \tilde v(x_m(\omega),\omega)= \tilde v(x_m).
\qedhere
\]
\end{proof}
\begin{proof}[End of the proof of Proposition \ref{prop:lambda_dot>0}]
We have three possibilities.

First, assume that $0\le \omega < \omega^*$ and that alternative \ref{i:v1} in Lemma \ref{lem:cases_v}
holds true. Then $v'(\pm\alpha)>0$, while we know from point \ref{i:1} of Lemma \ref{lem:apriori+subsup} that 
$u'(\alpha)<0< u'(-\alpha)$. Then we can use \eqref{eq:lambda(omega)}, \eqref{eq:lambda_dot} to write
\begin{equation*}
\dot \lambda = \frac{u'(-\alpha)v'(\alpha)-v'(-\alpha)u'(\alpha)}{u'(\alpha)^2}>0.
\end{equation*}

On the other hand, assume that $0\le \omega < \omega^*$ and that alternative \ref{i:v2} in Lemma 
\ref{lem:cases_v} holds true. By definition of $\tilde u$, $\tilde v$ we obtain 
\[
\tilde u'(-\alpha) + \lambda \tilde u'(\alpha)=0
\]
and, differentiating,
\[
\tilde v'(-\alpha) + \lambda \tilde v'(\alpha) + \dot \lambda \tilde u'(\alpha)=0.
\]
Now, $\tilde u'(\pm\alpha)$ has the same sign of $u'(\pm\alpha)$, while $\tilde v'(\pm\alpha)>0$ by Lemma 
\ref{lem:cases_vtilde}. Then we can rewrite $\dot \lambda$ in function of the rescaled variables, obtaining
\begin{equation*}
\dot \lambda = \frac{\tilde u'(-\alpha) \tilde v'(\alpha) - \tilde u'(\alpha) \tilde v'(-\alpha) }{\tilde u'(\alpha)^2}>0
\end{equation*}
also in this case.

Finally, in case $-\omega^*<\omega<0$, \eqref{eq:symm_lambda} yields
\[
\dot\lambda(\omega)=\lambda(\omega)^2\cdot\dot\lambda(-\omega)>0,
\]
since $0<-\omega<\omega^*$ falls within the cases we have already considered.
\end{proof}

\section{Bifurcation analysis}\label{sec:bif}

As we mentioned in Remark \ref{rmk:after_prop2}, to conclude the proof of our main results 
we are left to evaluate the limits
\[
\lim_{\omega\to\pm\omega^*}\lambda(\omega).
\]
This will be done through a bifurcation analysis, since we are going to show that the curve of 
positive solutions 
\[
(-\omega^*,\omega^*) \ni \omega \mapsto u(\cdot,\omega)
\]
bifurcates from the trivial solution at $\omega =\pm\omega^*$. Actually, as in the previous section, 
also here we can focus on e.g.~$\omega=\omega^*$, as the behaviour at $-\omega^*$ can be easily deduced 
by the symmetry properties \eqref{eq:symm_u}, \eqref{eq:symm_lambda}.

We will use well-established results about the bifurcation from simple eigenvalues, 
which allow a fairly detailed description of $u(\cdot,\omega)$ near $\omega = \omega^*$. In 
particular, our main references in this direction are \cite{MR288640} and \cite{AP}. For the reader's 
convenience, we summarize in the next statement a partial selection of the results therein, adapted to 
our purposes.

\begin{theorem}[{\cite[Theorem 1.7]{MR288640},\cite[Section 5.4]{AP}}]\label{thm:bif}
Let $X$, $Y$ be Banach spaces, $\omega^*\in\R$, $0\neq u^*\in X$, $0\neq \psi \in X^*$, and $F\in C^1(\R\times X,Y)$ satisfy
\begin{itemize}
\item $F(\omega,0)=0$ for every $\omega\in\R$;
\item the mixed second partial derivative $F_{\omega u}$ exists and is continuous;
\item $\ker F_u(\omega^*,0)=\spann\{u^*\}$, $F_u(\omega^*,0)[X]= \ker \psi$;
\item $\hat A:=\langle\psi, F_{\omega u}(\omega^*,0)\rangle \neq0$.
\end{itemize}
If $W$ is any complement of $\ker F(\omega^*,0)$ in $X$, then there exist $\eps>0$, and continuous 
functions $\mu \from (-\eps, \eps)\to\R$, $\gamma \from (-\eps, \eps)\to W$ 
such that $\mu(0)=0$, $\gamma(0)=0$ and the set of nontrivial solutions of 
$F(\omega,u)=0$,
\begin{equation*}
    \mathcal{S} :=\{(\omega,u)\in \mathbb{R} \times X: F(\omega,u)=0,\ u\neq 0 \} ,
\end{equation*}
is parametrized, locally at $(\omega^*,0)$, by
\[
(\omega(t),u(t))=(\omega^*+\mu(t),t u^* + t\gamma(t))
\qquad\text{ with } |t|<\eps. 
\]
Moreover, if $\partial_{uu}F$ is also continuous, the functions $\mu$ and $\gamma$ are $C^1$, and, 
as $\omega\to\omega^*$:
\begin{itemize}
\item if $\hat B:=-\frac{1}{2\hat A}\langle\psi, F_{u u}(\omega^*,0)[u^*]^2\rangle \neq0$ then the bifurcating branch 
is parametrized by 
\[
u=\frac{\omega-\omega^*}{\hat B}\,u^* + o(\omega-\omega^*);
\]
\item if $F_{u u}(\omega^*,0)[u^*]^2=0$, $\partial_{uuu}F$ is also continuous, and
$\hat C:=-\frac{1}{6\hat A}\langle\psi, F_{u u u}(\omega^*,0)[u^*]\rangle \neq0$ then the bifurcating 
branch is parametrized by 
\[
u=\pm\left(\frac{\omega-\omega^*}{\hat C}\right)^{1/2}\cdot u^* + O(\omega-\omega^*).
\]
\end{itemize}
\end{theorem}

To apply the previous theorem we choose 
\[
\omega^* = \sqrt{4 a - \frac{\pi^2}{\alpha^2}},
\] 
as in \eqref{eq:omega*} and, as in Section \ref{sec:omega_fixed}, $X= C^{2}_0([-\alpha,\alpha])$, $Y=C([-\alpha,\alpha])$, and the $C^1$ operator 
\begin{equation*}%\label{eq:F}
    F(\omega,u):=u''+\omega u' + a u -g(u).
\end{equation*}
Of course, the equation $F(\omega,u)=0$ has the trivial solution $u=0$ for all $\omega \in \mathbb{R}$. Moreover $F_u(\omega,u)[v] = v''+\omega v' + a v - g'(u)v$ and, since $g'(0)=0$, 
\begin{equation*}
F_u(\omega,0)[v]= v'' + \omega v' + av,
\end{equation*}
which is a Fredholm map of index 0, by Fredholm's alternative. Differentiating with respect to $\omega$, 
we obtain that the mixed second partial derivative 
\[
F_{\omega u}(\omega,0)[v]= v'.
\]
is well-defined and continuous, for every $\omega$.

The map $F_u(\omega,0)$ is not 
invertible if and only if
\[
    \begin{cases}
        v''+ \omega v' + av =0 & \text{in } \ (-\alpha,\alpha)\\
        v(-\alpha)=v(\alpha)=0,
    \end{cases}
\]
has nontrivial solutions, which happens if and only if
\begin{equation*}
   \omega=\omega_{\pm k}:=\pm\sqrt{4 a - \frac{k^2\pi^2}{\alpha^2}}\qquad 
   \text{with }k\in\left\{n\in\N:1\le n <\frac{2\alpha\sqrt{a}}{\pi}\right\},
\end{equation*}
with nontrivial solution (up to multiplicative constants)
\[
v_{\pm k}(x)= e^{-\frac{\omega_{\pm k}}{2} x } \cos{\left(\frac{k\pi}{2\alpha}x\right)}.
\]
By \eqref{eq:hp_a_alpha}, $\omega_1=\omega^*$ is always admissible. (On the other hand, any $k\ge 2$ is a possible bifurcation point for changing sign solutions, and is not relevant to the study of positive solutions because of the uniqueness property in Proposition \ref{prop:main_both}.) Then, setting
\begin{equation*}
    u^*(x) := v_1(x) = e^{-\frac{\omega^*}{2} x } \cos{\left(\frac{\pi}{2\alpha}x\right)},
\end{equation*}
it follows that the kernel of $L$ is one-dimensional, spanned by $u^*$:
\begin{equation*}
    \ker F_u(\omega^*,0) = \{tu^* : t \in \mathbb{R}\}.
\end{equation*}
By the Fredholm alternative, the range of $F_u(\omega^*,0)$ is closed and has codimension one. This implies the 
existence of $\psi\in X^*$ such that:
\begin{equation*}
    F_u(\omega^*,0)[X]= \{ y  \in Y : \langle \psi, y\rangle =0\}.
\end{equation*}
To find an explicit expression of $\psi$, we notice that $y$ belongs to the range of $F_u(\omega^*,0)$ if and only if there exist $v \in C^{2}_0([-\alpha,\alpha])$ such that
\begin{equation}\label{testtt}
    v'' + \omega ^* v' + a v = y \quad \text{in } \ (-\alpha,\alpha).
\end{equation}
Testing equation \eqref{testtt} with the function
\begin{equation*}
     v_{-1}(x) = e^{\frac{\omega^*}{2} x } \cos{\left(\frac{\pi}{2\alpha}x\right)}
\end{equation*}
we obtain
\begin{equation*}
    \int_{-\alpha}^\alpha v'' v_{-1} \, dx + \omega^* \int_{-\alpha}^\alpha v' v_{-1} \, dx + a \int_{-\alpha}^\alpha v v_{-1} \, dx = \int_{-\alpha}^\alpha y v_{-1} \, dx.
\end{equation*}
Integrating by parts and observing that $v_{-1}$ solves $v''-\omega^*v'+av=0$ in $(-\alpha,\alpha)$, we conclude that
\begin{equation*}
    \int_{-\alpha}^\alpha y v_{-1} \, dx =0.
\end{equation*}
Therefore, the range of $F_u(\omega^*,0)$ is the kernel of $\psi$, defined as
\begin{equation*}
    \langle \psi, y \rangle : = \int_{-\alpha}^\alpha e^{\frac{\omega^*}{2} x } \cos{\left(\frac{\pi}{2\alpha}x\right)} y(x) \, dx,\qquad\text{for every }y\in Y.
\end{equation*}
Finally, since 
\[
F_{\omega u}(\omega^*,0)[u^*] = (u^*)',
\]
the last assumption of Theorem \ref{thm:bif} is verified, as
\begin{equation}\label{eq:hatA}
    \hat A:=\langle \psi, F_{\omega u}(\omega^*,0)[u^*]\rangle = \int_{-\alpha}^\alpha e^{\frac{\omega^*}{2} x } \cos{\left(\frac{\pi}{2\alpha}x\right)} \cdot \left(e^{-\frac{\omega^*}{2} x } \cos{\left(\frac{\pi}{2\alpha}x\right)} \right)' \, dx = -\frac{\omega^*}{2} \alpha\neq0.
\end{equation}

Then the first part of  Theorem \ref{thm:bif} applies and, locally at $(\omega^*,0)$, $\Scal$ is 
parameterized by
\[
(\omega(t),u(t))=(\omega^*+\mu(t),t u^* + t\gamma(t)),
\qquad |t|<\eps, 
\]
with $\mu,\gamma$ continuous and vanishing at $t=0$. In particular,
\[
\lim_{t\to 0} \frac{u(t)}{t} = u^* \qquad \text{ in }C^2([-\alpha,\alpha]).
\]
Since $u^*>0$ in $(-\alpha,\alpha)$, and the cone of positive functions has nontrivial interior part in the topology of $C^2([-\alpha,\alpha])$, we have that $u(t)>0$ in $(-\alpha,\alpha)$ for $t>0$ small enough. Then, 
the uniqueness part of Proposition \ref{prop:main_both} implies 
that $u(t) = u(\omega(t),\cdot)$, as defined in that proposition. Finally, we have the following result.
\begin{proposition}\label{prop:lambda*}
With the definitions in \eqref{eq:omega*} we have
$\lim_{\omega\to\omega^*} \lambda(\omega)= \lambda^*$.
\end{proposition}
\begin{proof}
In view of the previous discussion, we have
\[
\lim_{\omega\to\omega^*} \lambda(\omega) = \lim_{\omega\to\omega^*} -\frac{u'(-\alpha,\omega)}{u'(\alpha,
\omega)}= 
%\lim_{t\to0} -\frac{u'(-\alpha,\omega(t))}{t}\cdot\frac{t}{u'(\alpha,\omega(t))}=
\lim_{t\to0} -\frac{u'(-\alpha,\omega(t))}{u'(\alpha,\omega(t))}=
-\frac{(u^*)'(-\alpha)}{(u^*)'(\alpha)}=e^{\omega^*\alpha}. \qedhere
\]
\end{proof}
\begin{proof}[End of the proofs of Theorems \ref{thm:main1}, \ref{thm:main2}]
All the properties in the theorems are proved in Proposition \ref{prop:main_both}, using as a parameter
$\omega$ instead of $\lambda$. The fact that we can equivalently use $\lambda$ as a parameter is a 
consequence of Proposition \ref{prop:lambda_dot>0}, while the thresholds for such parameter are provided 
by Proposition \ref{prop:lambda*} and equation \eqref{eq:symm_lambda}.
\end{proof}

We conclude this discussion with a more detailed analysis of the bifurcation diagram. Indeed, assuming 
more regularity on the nonlinearity $f(s)=as-g(s)$, it is possible to apply the second part of 
Theorem \ref{thm:bif} to describe the behavior of the bifurcating branch near $(\omega^*,0)$. 

If $g\in C^2$ and $g''(0) \neq 0$ (e.g.~in the logistic case $f(s)=as-bs^2$), we have that
the bifurcating branch can be parametrized, for $|\omega - \omega^*|$ small, by
\begin{equation}\label{eq:bif_logistic}
u=\frac{\omega-\omega^*}{\hat B}\,u^* + o(\omega-\omega^*);
\end{equation}
where, recalling \eqref{eq:hatA},
\begin{equation*}
    \hat B= -\frac{1}{2\hat A} \langle \psi, F_{u,u}(\omega^*,0)[u^*,u^*] \rangle= g''(0)
    \cdot h(\omega^*,\alpha) \neq 0,
\end{equation*}
and
\begin{equation*}
    h(\omega^*,\alpha) = -\frac{3 \pi}{\omega^*} \cosh{\left(\frac{\omega^*\alpha}{2}\right)}\left[\frac{1}{(\omega^*)^2\alpha^2+ \pi^2} - \frac{1}{(\omega^*)^2\alpha^2+ 9\pi^2}\right] < 0.
\end{equation*}
This follows since, in our case, we have
\begin{equation*}
    F_{u,u}(\omega^*,0)[u^*,u^*] = -g''(0) (u^*)^2,\quad
    \langle\psi, F_{u u}(\omega^*,0)[u^*]^2\rangle =- g''(0) \int_{-\alpha}^\alpha e^{-\frac{\omega^*}{2} x } \cos^3{\left(\frac{\pi}{2\alpha}x\right)} \, dx.
\end{equation*}
Moreover, the assumption $g>0$ yields $g''(0)>0$, whence $\hat B<0$. This implies that positive solutions of 
$F(\omega,u)=0$ branch off from the trivial one when $\omega < \omega^*$.

If instead $g''(0)=0$, but $g \in C^3$ at least in a neighborhood of $s=0$, and $g'''(0) \neq 0$ 
(e.g.~for $f(s)=as-bs^3$), we have that the bifurcating branch can be written in the form
\begin{equation*}
    u = \pm \left(\frac{\omega-\omega^*}{\hat C}\right)^\frac{1}{2}u^* + O(\omega-\omega^*),
\end{equation*}
where
\begin{equation*}
    \hat C:= -\frac{1}{6\hat A} \langle \psi, F_{uuu}(\omega^*,0)[u^*]^3 \rangle= g'''(0)\cdot
    \bar{h}(\omega^*,\alpha) \neq 0,
\end{equation*}
and
\begin{equation*}
    \bar{h}(\omega^*,\alpha) = -\frac{\sinh(\omega^*\alpha)}{3\omega^*\alpha} \left[\frac{3}{4\omega^*} + \omega^* \alpha^2\left(\frac{1}{4}\frac{1}{\alpha^2(\omega^*)^2 + 4\pi^2}- \frac{1}{\alpha^2(\omega^*)^2 + \pi^2} \right)\right] < 0.
\end{equation*}
Indeed, it is straightforward to verify that
\begin{equation*}
    F_{uuu}(\omega^*,0)[u^*]^3 = -g'''(0) (u^*)^3,\quad
     \langle\psi, F_{u u u}(\omega^*,0)[u^*]^3\rangle=  -g'''(0) \int_{-\alpha}^\alpha e^{-\omega^* x } \cos^4{\left(\frac{\pi}{2\alpha}x\right)} \, dx.
\end{equation*}
The assumption $g>0$ yields $g'''(0)>0$. Hence $\hat C <0$ and, also in this case, 
the bifurcation branch emanates to the left of $\omega^*$, in accordance with our main results.

\begin{remark}\label{rmk:altern_2}
As a byproduct of the bifurcation analysis, we have that, at least when $g$ is $C^2$ and $g''(0)>0$ and 
in view of \eqref{eq:def_v} and \eqref{eq:bif_logistic},  
\[
v:=\frac{\partial}{\partial \omega} u =\frac{1}{\hat B}u^* + o(1) \qquad\text{ as }\omega\to\omega^*,
\]
where $\hat B<0$. Then, alternative \ref{i:v2} in Lemma \ref{lem:cases_v} is verified when $\omega$ is near 
$\omega^*$. As we already noticed in Remark 
\ref{rmk:cases_v}, since instead if $\omega=0$ then alternative \ref{i:v1} is verified, 
by continuity there exists at least a value of $\omega$ such that alternative \ref{i:v2} is verified with $v'(-\alpha)=0$.
\end{remark}

\begin{remark}\label{rmk:global_rab}
Although this has not an impact in our analysis, we notice that also global bifurcation results can be 
applied to our problem, in particular those in \cite{MR301587}. As a consequence, we have that the branch 
$\omega\mapsto u(\cdot,\omega)$ is an interesting example of a global branch which is not unbounded,
and 
thus it meets the trivial solution at two distinct bifurcation points, i.e.~$(\pm\omega^*,0)$.
\end{remark}

\textbf{Data Availability.} Data sharing not applicable to this article as no datasets were generated or analyzed during the current study.

\bigskip

\textbf{Competing Interest.} The authors report there are no competing interests to declare.

\bigskip

\textbf{Funding Declaration.} Work partially supported by: PRIN-20227HX33Z ``Pattern formation in 
nonlinear phenomena'' - funded  by the European Union-Next Generation EU, Miss. 4-Comp. 1-CUP 
D53D23005690006 and PRIN PNRR P2022YFAJH ``Linear and Nonlinear PDEs: New directions and applications''. GS and GV are members of the INdAM-GNAMPA group (``Gruppo Nazionale per 
l'Analisi Matematica, la Probabilit\`a e le loro Applicazioni -- Istituto Nazionale di Alta
Matematica'').

%-------------------------------------------------------------------------
%	BIBLIOGRAPHY
%-------------------------------------------------------------------------

\bibliographystyle{abbrv}
\bibliography{LV.bib}

\medskip
\small
\begin{flushright}
\noindent 
\verb"giuseppe.spadaro@unical.it"\\
Dipartimento di Matematica e Informatica, Università della Calabria\\ 
Ponte Pietro Bucci 31B, 87036 Arcavacata di Rende, Cosenza (Italy)\medskip\\		
\noindent 
\verb"gianmaria.verzini@polimi.it"\\
Dipartimento di Matematica, Politecnico di Milano\\ 
piazza Leonardo da Vinci 32, 20133 Milano (Italy)\medskip\\
\noindent
\verb"alessandro.zilio@u-pariscite.fr"\\
CNRS, Laboratoire Jacques-Louis Lions (LJLL), \\
Université Paris Cité and Sorbonne Université,
75005 Paris (France)
\end{flushright}

\end{document}